\documentclass[12pt]{article}

\usepackage{fullpage}
\usepackage{amsthm,amssymb}
\usepackage{graphics,graphicx}
\usepackage{amsmath}
\usepackage{hyperref}
\newtheorem{lem}{Lemma}

\newtheorem*{mainthm}{Main Theorem}
\newtheorem{thm}[lem]{Theorem}
\newtheorem{conjecture}[lem]{Conjecture}

\def\mad{{\rm{mad}}}
\def\config#1{\item[{\rm (#1)}]}
\def\floor#1{\left\lfloor #1\right\rfloor}
\def\l{\ell}
\def\chil{\chi_{\ell}}
\usepackage{epstopdf}
\begin{document}
\title{
Sufficient sparseness conditions for $G^2$ to be\\ $(\Delta+1)$-choosable, when $\Delta\ge
5$.
}
\author{
Daniel W. Cranston\thanks{Virginia Commonwealth University, Department of
Mathematics and Applied Mathematics, Richmond, VA, USA. Email: {\tt
dcranston@vcu.edu}}
\and
Riste \v Skrekovski\thanks{Department of Mathematics, University of Ljubljana, Ljubljana 
\& Faculty of Information Studies, Novo Mesto, Slovenia.  Partially supported by ARRS Program P1-0383. 
Email: {\tt skrekovski@gmail.com}
}
}
\date{\today}
\maketitle

\noindent
{\bf Keywords:} choosability; square of a graph;
maximum average degree; discharging; girth; maximum degree \\

\abstract{We determine the list chromatic number of the square of a
graph $\chil(G^2)$ in terms of its maximum degree $\Delta$ when its maximum
average degree, denoted $\mad(G)$, is sufficiently small.
For $\Delta\ge 6$, if $\mad(G)<2+\frac{4\Delta-8}{5\Delta+2}$, then
$\chil(G^2)=\Delta+1$.  In particular, if $G$ is planar with girth $g\ge
7+\frac{12}{\Delta-2}$, then $\chil(G^2)=\Delta+1$.
Under the same conditions, $\chil^i(G)=\Delta$, where $\chil^i$ is the list
injective chromatic number.
}

\section{Introduction}

The square of a graph $G$, denoted by $G^2$, is the graph with $V(G^2)=V(G)$ and
$E(G^2)=\{uv\mid  d_G(u,v)\leq 2\}$; in other words, two vertices are adjacent
in $G^2$ if they are at distance at most two in $G$.
If $G$ has maximum degree $\Delta$, then coloring $G^2$
requires at least $\Delta+1$ colors; the upper bound $\Delta^2+1$
follows from  the greedy algorithm.  This upper bound is also achieved for a
few graphs, for example for the 5-cycle and the Petersen graph.

Regarding the coloring of squares of planar graphs, Wegner~\cite{wegner} posed
the following central problem.

\begin{conjecture}[Wegner]
\label{wegner}
For a planar graph $G$ of maximum degree $\Delta$:
$$\chi(G^2)\leq \left\{  \begin{array}{ll}
         7, & \mbox{$\Delta=3$};\\
         \Delta+5, & \mbox{$4\leq \Delta \leq 7$};\\
         \lceil\frac{3}{2}\Delta \rceil+1, & \mbox{$\Delta \geq 8$}.
         \end{array} \right.$$
\end{conjecture}

In~\cite{Havet1} Havet, van den Heuvel, McDiarmid, and Reed showed that 
the following holds: $\chi(G^2)\leq \frac{3}{2}\Delta(1+o(1))$,
which is also true for the choice number (defined below).
Dvo\v r\'ak, Kr\'al', Nejedl\'y, and \v Skrekovski~\cite{skreko1} showed that
the square of every planar graph of girth at least six with sufficiently large
maximum degree $\Delta$ is $(\Delta+2)$-colorable.
Borodin and Ivanova~\cite{borodin2} strengthened this result to prove that for
every planar graph $G$ of girth at least six with maximum degree
$\Delta \ge 24$, the choice number of $G^2$ is at most $\Delta+2$.
Most recently, Bonamy, L\'eveque, and Pinlou~\cite{Bonamy} showed the same
conclusion for every planar graph $G$ with girth at least six and $\Delta\ge 17$.
In fact, their proof only requires $\mad(G)<3$ (defined below).
Lih, Wang, and Zhu~\cite{Lih} showed that the square of a $K_4$-minor free
graph with maximum degree $\Delta$ has chromatic number at most $\lfloor
\frac{3}{2}\Delta \rfloor+1$ if $\Delta \geq 4$ and $\Delta+3$ if $\Delta
\in\{2,3\}$.
Hetherington and Woodall~\cite{Hetherington} showed that the bounds
in~\cite{Lih} also hold for the choice number.

We write $\Delta$ for the maximum degree of a fixed graph $G$.
A {\it $k$-vertex} is a vertex of degree $k$.  Similarly, a {\it $k^+$-vertex}
(resp.\ {\it $k^-$-vertex}) is a vertex of degree at least (resp.\ at most) $k$.
 A {\it $k$-thread} is a path with $k$ internal $2$-vertices.
The {\it endpoints} of a thread are its first and last vertices.  A {\it weak
neighbor} of a vertex $v$ is one joined to $v$ by a $k$-thread (for some $k\ge
1$), and a {\it weak $k$-neighbor} is a weak neighbor that is a $k$-vertex.
We write $N(v)$ for the neighborhood of $v$ and $N[v]$ for $N(v)\cup\{v\}$.

A proper {\it coloring} of the vertices of a graph $G$ is
a mapping $c:V(G)\rightarrow \mathbb{N}$  such that every two adjacent vertices are
mapped to different colors.
Elements of $\mathbb{N}$ are {\it colors}.
List coloring was first studied by Vizing~\cite{vizing}
and is defined as follows.  Let $G$ be a simple graph. A {\it list-assignment}
$L$ 
is an assignment of lists of colors to vertices.  A {\it list-coloring} is
a coloring where each vertex $v\in V(G)$ receives a color from $L(v)$, and the graph
$G$ is {\it $L$-choosable} if there is a proper $L$-list-coloring.
If $G$ has a list-coloring for every list-assignment with $|L(v)|\ge k$ for each
vertex $v$, then $G$ is {\it $k$-choosable}.
The minimum $k$ such that $G$ is $k$-choosable is the {\it choice number}
of $G$, and is denoted by $\chi_l$.
An {\it injective coloring} of a graph $G$ is a mapping $c:V(G)\rightarrow
\mathbb{N}$  such that vertices with a common neighbor are mapped to different
colors (but it need {\it not} be proper).  The {\it injective chromatic number}
$\chi^i(G)$ and {\it injective choice number} $\chil^i(G)$ are defined
analogously.  For each $G$, we have $\chil^i(G)\le\chil(G^2)$.

In the proofs of our theorems, we use the discharging method, which was first used
by Wernicke~\cite{wernicke}, and which is most well-known for its central role in the
proof of the Four Colour Theorem.  Here we apply the discharging method in the
more general context of bounded {\it maximum average degree}, denoted $\mad(G)$,
which is defined as $\mad(G):=\max_{H\subseteq G}\frac{2|E(H)|}{|V(H)|}$,
where $H$ ranges over all subgraphs of $G$.
A straightforward consequence of Euler's Formula is that every planar graph $G$
with girth at least $g$ satisfies $\mad(G)<\frac{2g}{g-2} = 2+\frac{4}{g-2}$.
Using this bound on $\mad$, our results for planar graphs follow immediately from
corresponding results for maximum average degree.
The key tool in many of our proofs is global discharging, which relies on
reducible configurations that may be arbitrarily large.  Global discharging was
introduced by Borodin~\cite{borodin1}, and has been applied widely; for example,
see~\cite{borodin2}  and~\cite{cranston1}.

Kostochka and Woodall~\cite{kostochka} conjectured that every square of a
graph has choice number equal to chromatic number, i.e., $\chi_l(G^2)
=\chi(G^2)$.  This conjecture inspired much research, although recently it has
been disproved~\cite{Kim}.
For planar graphs, the best upper bound on $\chi(G^2)$ in terms of $\Delta$ was
succesively improved
by Jonas~\cite{Jonas}, Wong~\cite{Wong}, Van den Heuvel and McGuinness~\cite{Heuvel},
 Agnarsson and Halld´orsson~\cite{Agnarsson}, Borodin et al.~\cite{borodin} and finally by
 Molloy and Salavatipour~\cite{Molloy} to the best known upper bound so far,
$\chi(G^2)\leq \left\lceil \frac{5}{3}\Delta\right\rceil+78$.
For the best asymptotic upper bound, see~\cite{Havet1}, mentioned above.

The choosability of squares of subcubic planar graphs has been extensively
studied by Dvo\v r\'ak, \v Skrekovski, and Tancer~\cite{skreko}, Montassier and
Raspaud~\cite{montassier}, Havet~\cite{havet}, and Cranston and
Kim~\cite{cranston}. In \cite{CES}, we gave upper bounds on $\chi_l(G^2)$ when
$\Delta(G)= 4$ and $\mad(G)$ is bounded. In the present paper, we again
consider graphs $G$ with bounded maximum average degree, but now with
higher maximum degree.  For $\Delta(G)\ge 6$, our results are summarized in
the following theorem.

\begin{mainthm} Let $G$ be a graph with maximum degree $\Delta\ge 6$.
If $\mad(G)<2+\frac{4\Delta-8}{5\Delta+2}$, then $\chil(G^2)= \Delta+1$.  In particular,
if $G$ is planar with girth $g\ge 7+\frac{12}{\Delta-2}$, then $\chil(G^2)=\Delta+1$.
\end{mainthm}

Besides our Main Theorem, for $\Delta=5$ we prove that $\mad(G)<2+12/29$
implies $\chil(G^2)=6$.  Note that for $\Delta=4$, in \cite{CES} we proved
that $\mad(G)<2+2/7$ implies $\chil(G^2)=5$.
We also construct examples with maximum degree $k$ and $\mad$ arbitrarily close
to $2+2/7$ (resp.\ $2+12/29$ and $2+(4k-8)/(5k+2)$) that contain none of the
reducible configurations we use in the proofs.  So to improve the coloring
results, we need additional reducible configurations.

The Main Theorem is proved in three parts: $k=6$, $k=7$, and $k\ge 8$.
In each part, we assume a counterexample with the fewest vertices, then
reach a contradiction.  When we remove one or more vertices from this graph,
the square of the result can be properly colored from its lists. We elaborate
on this approach in the next section.

We mention in passing that each time that we prove that $\chil(G^2)=\Delta+1$,
the proof can be modified to show that $\chil^i(G)=\Delta$.  The coloring
algorithms are the same, but now each vertex has at least one fewer constraints
on its color.
\bigskip

After submitting this paper, we learned that Bonamy, L\'{e}v\^{e}que, and
Pinlou~\cite{BLP} have submitted a paper proving similar 
results, using similar methods.  For $\Delta\in\{6,8\}$, their results match
ours.  In general, for $\epsilon > 0$, they proved there exists
$\Delta_{\epsilon}$ such that if $\mad(G)<\frac{14}5-\epsilon$ and $\Delta\ge
\Delta_{\epsilon}$, then $\chil(G^2)=\Delta+1$ and $\chil^i(G)=\Delta$.
However, for general $\epsilon$ (other than
$\epsilon\in\{\frac3{10},\frac8{35}\}$, which
corresponds to $\Delta\in\{6,8\}$), their bound on $\Delta_{\epsilon}$ is
slightly weaker than ours. 
Perhaps the best comparison is as follows.   To prove that
$\chil(G^2)=\Delta+1$ and $\chil^i(G)=\Delta$ when
$\mad(G)<\frac{14}5-\epsilon$, their required lower bound on $\Delta$ is a
little more than $\frac53$ times as large as ours.

%
%
%

\section{Reducible configurations}
\label{sect:configs}

A {\it configuration} is an induced subgraph of a graph $G$. A
configuration $C$ is {\it reducible} if whenever $G$ contains $C$, we can form a
graph $G'$ with fewer edges than $G$ such that any good coloring of $G'$ gives
rise to a good coloring of $G$.
Thus a reducible configuration cannot appear in a minimal counterexample.
For convenience, we often write {\it color $G-C$} to mean color $(G-C)^2$ from
its assigned lists.
To prove that a configuration is reducible, we infer from the minimality of
$G$ that the subgraph $G-C$ can be properly colored, and then prove that this
coloring can be extended to a proper coloring of the original graph $G$, which
gives a contradiction.

A configuration is {\it $k$-reducible} if it is reducible in the setting of
$k$-choosability. Clearly a $k$-reducible configuration is also
$(k+1)$-reducible. 
Perhaps the easiest example is that a $1$-vertex $v$ is $(\Delta+1)$-reducible,
since $G-v$ can be colored by minimality, and $v$ has at most $\Delta$
neighbors in the square, each forbidding at most one color.
Here we show the reducibility of some of the configurations
that we use later in the proof.  For consistency with our application of these
lemmas later on (when $k$ denotes the maximum degree of a vertex), here we
prove that configurations are $(k+1)$-reducible.

\begin{figure}[bt]
\label{f.2}
\begin{center}
  \begin{tabular}{ccc}
     \includegraphics[scale=.95]{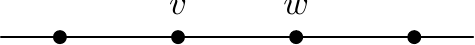}\hspace{.4cm}&
     \includegraphics[scale=.95]{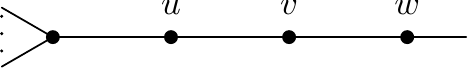}\hspace{.4cm}&
     \includegraphics[scale=.95]{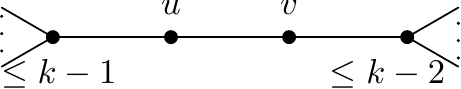}\\
     (C1) & (C2) & (C3)
  \end{tabular}
  \caption{Configurations from Lemma~\ref{l.2}}
\end{center}
\end{figure}

\begin{lem} \label{l.2}
For $k\ge 4$, the following configurations are $(k+1)$-reducible:
\begin{enumerate}
\config{C1} a 4-thread;
\config{C2} a 3-thread with an endpoint of degree at most $k-1$;
\config{C3} a 2-thread with endpoints of degree at most $k-1$ and $k-2$;
\end{enumerate}
\end{lem}
\begin{proof} For illustration see Figure~1. 
For simplicity we assume in (C1)--(C3) that the endpoints of the thread are
distinct.  This assumption can be justified by showing that any 3-cycle,
4-cycle, or 5-cycle with at most one $3^+$-vertex is $(k+1)$-reducible.  The
arguments are similar to those below (but even easier), so we omit them.

  The reducibility of (C1) is given in \cite{CES} but we repeat it here. 
Let
$u$ and $v$ be the middle two vertices of the 4-thread.  By the minimality of
$G$ we can color 
$G-u-v$.  
We say a color is \emph{available} for a vertex $v$, if it is in the list of
allowable colors for $v$ and has not already been used on neighbor of $v$ in
$G^2$.
Now $u$ and $v$ each have at least
two available colors, so we can easily extend the coloring to $G^2$.

Let $uvw$ be the 3-thread from (C2) with $u$ adjacent to an endpoint of
degree at most $k-1$. By minimality of $G$, we color $G-u-v$.  Now $u$ has at
least one available color, and $v$ has at least two.  So color first $u$ and
then $v$ to get a coloring of $G^2$.

Let $uv$ be the 2-thread from (C3) with $u$ and $v$ adjacent to endpoints of
degrees at most $k-1$ and $k-2$, respectively. If we color $G-u-v$, then $u$
has at least one available color and $v$ has at least two. So we can easily
extend this coloring to $G^2$.
\end{proof}

\begin{figure}[!b]
\label{f.3}
\begin{center}
  \begin{tabular}{ccc}
     \includegraphics[scale=0.85]{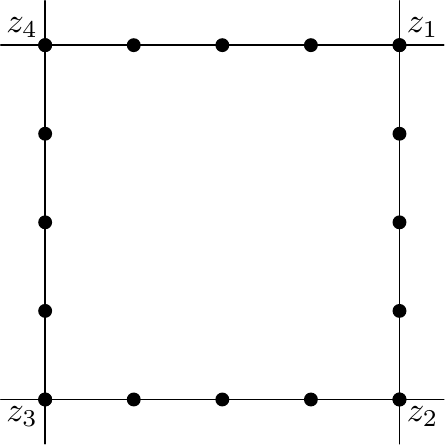}\hspace{.2cm}&
     \includegraphics[scale=.85]{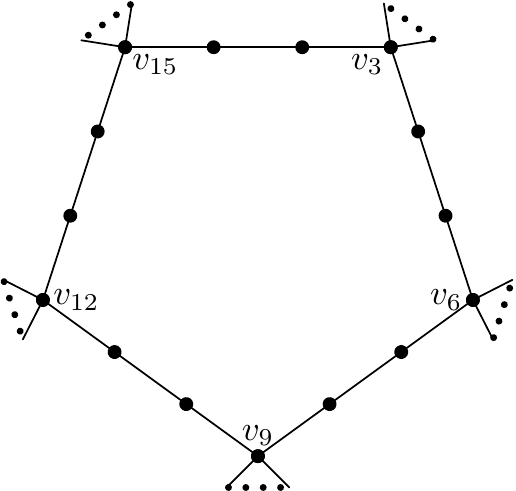}\hspace{.2cm}&
     \includegraphics[scale=.85]{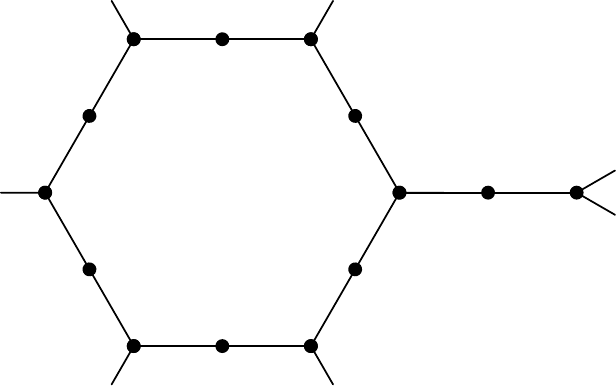}\\
     (C4) & (C5) & (C6)~~~~~~~~~~~
   \end{tabular}
      \caption{Configurations from Lemma~\ref{l.3}}
\end{center}
\end{figure}

\begin{lem} \label{l.3}
For $k\ge 5$, the following configurations are $(k+1)$-reducible:
\begin{enumerate}
\config{C4} a $(4\l)$-cycle $v_1v_2\dots v_{4\l}$ such that $d(v_i)\le k$
when $4\,|\,i$ and $d(v_i)=2$ otherwise;
\config{C5} a $(3\l)$-cycle $v_1v_2\ldots v_{3\l}$ such that $d(v_i)\le k-1$
when $3\,|\,i$ and $d(v_i)=2$ otherwise;
\config{C6} a cycle induced by 1-threads incident to 3-vertices (at both
ends) with at least one of these 3-vertices on the cycle incident to a third
1-thread with a 3-vertex at the other end.
\end{enumerate}
\end{lem}
\begin{proof} See Figure~\ref{f.3} for illustration of configurations (C4)--(C6).
Remove all 2-vertices of (C4), and by minimality color the square of the
resulting graph. Now the subgraph of $G^2$ induced by all $v_i$ with $i$ odd is
a $2\l$-cycle.  Each of these vertices has at least two available colors (and
even cycles are 2-choosable~\cite{ERT}), so we
extend the coloring to them.  Finally, we color the $v_i$ with $4\,|\,(i+2)$,
each of which has an available color.

Remove all 2-vertices of (C5), and by minimality color the square of the
resulting graph. Now the subgraph of $G^2$ induced by all uncolored $v_i$ is
a $2\l$-cycle.  Each of these vertices has at least two available colors, so we
extend the coloring to $G^2$.

In (C6), let $v$ be a 3-vertex on the cycle with 3 incident 1-threads leading to
3-vertices.  By minimality, we can color $(G\setminus N[v])^2$.
Now uncolor all 2-vertices on the cycle and color $v$.
%
Notice that in $G^2$ the 2-vertices of (C6) induce two cycles that share an
edge, and one of these two cycles has length 3. Moreover the two vertices
that belongs to both cycles have at least 3 available colors, and all others
have at least 2 colors. By Vizing's degree-choosability
theorem~\cite{vizing}, we can extend this coloring to these vertices.
We should also mention the case where the 2-neighbor $u$ of $v$ not on the
cycle has its other 3-neighbor also on the cycle.  In this case, in $G^2$ the
2-vertices induce a cycle with one additional vertex adjacent to four cycle
vertices.  Again we can complete the coloring by Vizing's degree-choosability
theorem.
\end{proof}

Now we construct examples to show that 
%
the threshold $2+\frac{4k-8}{5k+2}$ in the Main Theorem cannot be
improved without adding new reducible configurations (or taking a completely
different approach for the proof).
We construct examples with maximum degree $k$ and $\mad$ arbitrarily close to
$2+\frac{4k-8}{5k+2}$ that do not contain any of the above reducible configurations.
(In fact, our proof of the Main Theorem does use some additional reducible
configurations of bounded size, but none of them appear in our examples either.)
Example 1 is tight for $\Delta\in\{4,5\}$ and Example 2 is tight for $\Delta\ge 6$.

\paragraph{Example 1.} Let $G$ be a bipartite graph with vertices in part $A$
of degree $k-2$ and vertices in part $B$ of degree $k-3$. Subdivide each edge
of the graph twice.  Now add a spanning cycle $C_1$ through the vertices of $A$
and a spanning cycle $C_2$ through the vertices of $B$.  Subdivide each edge of
$C_1$ three times, and subdivide each edge of $C_2$ twice.
The average degree of this graph is $3-(7k-18)/(2k^2-3k-6)$, which is $2+2/7$
and $2+12/29$ for $k=4$ and $k=5$ respectively.  However, if we contract just
one edge on what was $C_1$ and one edge on what was $C_2$, we get a graph with
none of the above reducible configurations.

\smallskip

\paragraph{Example 2.}
Begin with a $(k-2)$-regular graph on a set $A$ of $2M$ vertices (for arbitrary
fixed $M$) and an independent set $B$ of size $M$.
Subdivide each edge incident to $A$ five times and add one edge from the center
vertex of each resulting 5-thread to a vertex of $B$ so that each vertex of $B$
now has degree $k-2$.  Add a spanning cycle $C_1$ through the vertices of $A$
and a spanning cycle $C_2$ through the vertices of $B$.  Finally, subdivide
each edge of $C_1$ and $C_2$ three times.  The average degree of this graph is
$2+(4k-8)/(5k+2)$.  If we contract one edge each on what was $C_1$ and $C_2$,
the resulting graph has none of the reducible configurations above.  Further,
it contains only vertices of degrees 2, 3, and $k$.

\smallskip

We suspect that in this way we can construct graphs with arbitrarily high
girth--probably we can adapt the construction of regular graphs with
arbitrary degree and arbitrary girth. If so, then any set of reducible
configurations that appears in all graphs formed by this construction must
contain new arbitrarily large reducible configurations.

\section{Maximum degree 5}

\begin{thm}
\label{T4}
If $\Delta\le 5$ and $\mad(G)<12/29$, then $\chil(G^2)\le 6$.
\end{thm}
\begin{proof}
Assume that the theorem is false and let $G$ be a minimal counterexample.
Note the following properties of $G$. By reducible configurations (C1) and
(C2), $G$ contains no 4-thread and each 3-thread has both endpoints of degree
5.  By configuration (C4), the subgraph induced by 3-threads is
acyclic.  Hence, we can assign each $5$-vertex to sponsor at most one incident
3-thread so that every 3-thread is sponsored.
Let $F$ denote the subgraph induced by vertices incident to (or on) 3-threads.  Let $v$
be a 5-vertex that is a leaf in $F$.  Assign $v$ to sponsor its incident
3-thread; now delete $v$ and its 3-thread, and recurse.

Now we use reducibility of (C5) and (C6). Similar to 5-vertices sponsoring 3-threads,
we assign to each 2-thread with 4-vertices at both ends an incident 4-vertex to
sponsor it.  Likewise, we assign to each 1-thread with 3-vertices at both ends an
incident 3-vertex to sponsor it.

\begin{figure}
\label{f.4}
\begin{center}
\begin{tabular}{cc}
   \includegraphics[scale=1]{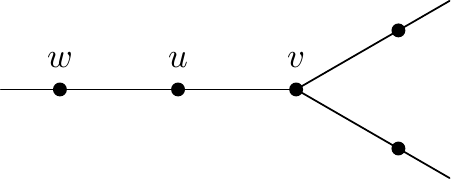}\hspace{.7cm}&
   \includegraphics[scale=1]{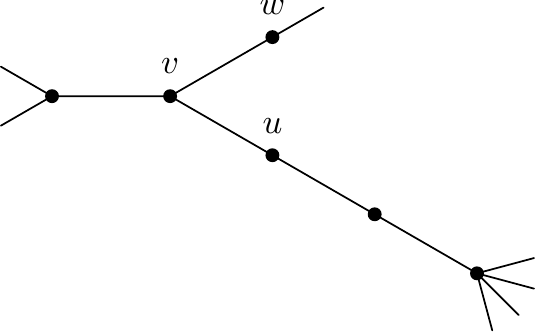}\\
    $(i)$ & $(ii)$
\end{tabular}
  \caption{Reducible configurations from Theorem~\ref{T4}}
\end{center}
\end{figure}

We use discharging with each vertex $v$ getting initial charge $d(v)$ and with
the following discharging rules.
Charge sent to a thread will be equally distributed among the 2-vertices of that
thread.
\begin{enumerate}
\item[(R1)]  Every 5-vertex sends charge $13/29$ to each incident thread or
adjacent 3-vertex, and it sends extra charge $10/29$ to its sponsored 3-thread
(if it exists).
\item[(R2)] Every 4-vertex sends charge $11/29$ to each incident thread or
adjacent 3-vertex, and it sends extra charge $2/29$ to its sponsored 2-thread
(if it exists).
\item[(R3)] Every 3-vertex sends charge $11/29$ to each incident 2-thread, it
sends charge $1/29$ to each incident 1-thread leading to a weak 4-neighbor, and
also it sends charge $12/29$ to its sponsored 1-thread (if it exists).
\end{enumerate}
Now we show that every vertex $v$ finishes with charge $\mu^*(v)$ at least $2+12/29$.
If $d(v)=5$, then $\mu^*(v)\ge 5-5(13/29)-(10/29)=2+12/29$. If $d(v)=4$, then
$\mu^*(v)\ge 4-4(11/29)-(2/29)=2+12/29$. Now we consider the two remaining
possibilities $d(v)=3$ and $d(v)=2$.

Suppose $d(v)=3$. Consider the possibility that $v$ is incident with three
threads. If all of them are 1-threads, then $\mu^*(v)\ge 3 - 12/29-2(1/29)>
2+12/29$.  If at least one is a 2-thread, then let the 2-thread be $uw$,
with $u$ adjacent to $v$. See Figure~\ref{f.4}$(i)$.
Remove $u$ and by minimality color $G-u$. Now recolor $v$ with a distinct
color from $w$ if necessary, and then color $u$.
So $v$ has a $3^+$-neighbor.

If $v$ has a $4^+$-neighbor, then $\mu^*(v)\ge
3+11/29-11/29-12/29> 2+12/29$.  Similarly, if $v$ has at least two
$3$-neighbors, then $\mu^*(v)\ge 3-12/29 > 2+12/29$.  So $v$ must have a
3-neighbor and two 2-neighbors.  If $v$ has a 3-neighbor and two incident
1-threads, then $\mu^*(v)\ge 3-12/29-1/29> 2+12/29$.  So $v$ must have
an incident 2-thread
and either another incident 2-thread or else an incident 1-thread leading to a
weak 3-neighbor (if it leads to a weak $4^+$-neighbor, then $v$ gives away at
most $12/29$).
Let $u$ be a neighbor of $v$ on a 2-thread and let $w$ be
$v$'s other 2-neighbor. See Figure~\ref{f.4}$(ii)$. By minimality, we can color
$G\setminus\{u,v,w\}$.  Now we color $v$, $w$, and $u$ in this order.

Finally, suppose $d(v)=2$.  We show that each $\l$-thread $P$ finishes with
charge at least $\l(2+12/29)$, so that each 2-vertex finishes with charge at
least $2+12/29$.  If $\l=3$, then $P$ gets charge $13/29$ from each endpoint
and charge $10/29$ from its sponsor, so $\mu^*(P)\ge
6+2(13/29)+10/29=6+36/29=3(2+12/29)$.  Suppose $\l=2$.  If $P$ has a 5-vertex as
an endpoint, then $\mu^*(P)\ge 4+ 13/29+11/29=2(2+12/29)$.  If $P$ has two
4-vertices as endpoints, then $\mu^*(P)\ge 4 +2(11/29)+2/29=4+24/29=2(2+12/29)$.
Since (C3) is reducible, we are in one of these cases.
Finally, suppose $\l=1$.
If $P$ has a 5-vertex endpoint, then $\mu^*(P)\ge 2+13/29$.  If $P$ has a
4-vertex endpoint, then $\mu^*(P)\ge 2+11/29+1/29=2+12/29$.  If $P$ has two
3-vertex endpoints, then $P$ gets $12/29$ from its sponsor, so $\mu^*(P)\ge
2+12/29$.

So each vertex finishes with charge at least $2+12/29$.  This contradicts the
fact that $\mad(G)<2+12/29$, and thus completes the proof.
\end{proof}

%
%
%

\section{Maximum degree 6}

\begin{thm}
\label{T5}
If $\Delta\le 6$ and $\mad(G)<5/2$, then $\chil(G^2)\le 7$.
\end{thm}
\begin{proof}
Assume to the contrary that the theorem is false and let $G$ be a minimal
counterexample.  A vertex is \emph{high} if its degree is 5 or 6, it is
\emph{medium} if its degree is 3 or 4, and it is \emph{low} otherwise.  By
reducible configurations (C1) and (C2), $G$ contains no 4-thread and each
3-thread has both endpoints of degree 6.  By configuration (C4), the
subgraph induced by 3-threads is acyclic.  Hence, we can assign each $6$-vertex
to sponsor at most one incident 3-thread so that every 3-thread is sponsored.

\begin{figure}[b]
\label{f.5}
\begin{center}
   \begin{tabular}{cccc}
     \includegraphics[scale=.95]{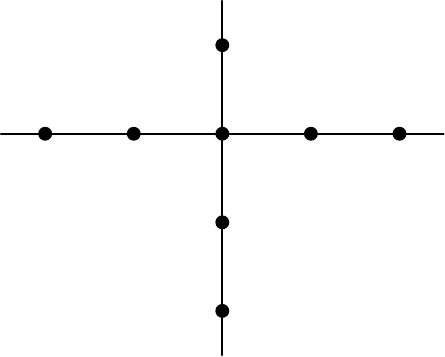}\hspace{.3cm}&
     \includegraphics[scale=.95]{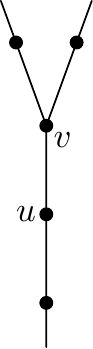}\hspace{.3cm}&
     \includegraphics[scale=.95]{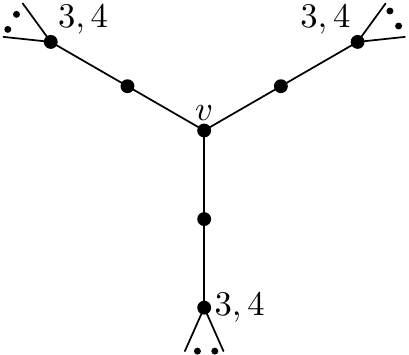}\hspace{.3cm}&
     \includegraphics[scale=.95]{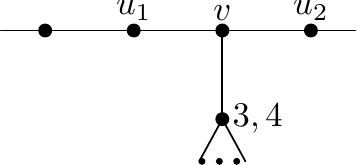}\\
       $(i)$ & $(ii)$ & $(iii)$  & $(iv)$
   \end{tabular}
  \caption{Reducible configurations from Theorem~\ref{T5}}
\end{center}
\end{figure}

We use discharging with each vertex $v$ getting initial charge $d(v)$ and with
the following discharging rules.
\begin{enumerate}
\item[(R1)] Every high vertex sends charge $1/2$
in each direction.
\item[(R2)]  Every 6-vertex sends extra charge $1/2$ to its sponsored 3-thread (if it
exists).
\item[(R3)]  Every medium vertex sends charge $1/2$ to each incident 2-thread,
and it sends charge $1/4$ to each incident 1-thread with other endpoint a medium vertex.
\end{enumerate}
We now show that every vertex $v$ finishes with final charge $\mu^*(v)$ at least
$5/2$, which will establish the theorem.

If $d(v)=6$, then $\mu^*(v)\ge 6-6(1/2)-1/2=5/2$.  If $d(v)=5$, then
$\mu^*(v)\ge 5-5(1/2)=5/2$.  If $d(v)=4$, then $v$ finishes with charge at
least $5/2$ unless $v$ gives charge in all 4 directions and gives charge $1/2$
in at least 3 of these directions; see Figure~\ref{f.5}$(i)$.  By minimality,
we can color $G\setminus N[v]$.  Now we easily color $v$'s neighbor on a
1-thread (if one exists), followed by $v$,
followed by its remaining neighbors. This produces a proper coloring
of $G^2$. So the cases left to consider are $d(v)=3$ and $d(v)=2$.

Suppose $d(v)=3$.  Consider first the possibility that $v$ has three
2-neighbors and is incident to at least one 2-thread.  Let $u$ denote its
neighbor on a 2-thread; see Figure~\ref{f.5}$(ii)$.  By minimality, we can
color $G-u$.  To extend the coloring to $G$, first recolor $v$ to
avoid the color on $u$'s other neighbor, then color $u$.  If $v$ has three
2-neighbors and at least one high weak neighbor, then $\mu^*(v)\ge 3 -
2(1/4)=5/2$.  So suppose $v$ has three 2-neighbors and all its weak neighbors
are medium. See Figure~\ref{f.5}$(iii)$.  By minimality, we color $G\setminus
N[v]$. Now color the vertices in $N(v)$ in arbitrary order;
finally, color $v$.  Thus, we may assume that $v$ has a $3^+$-neighbor.

If $v$ has a $5^+$-neighbor,
then $\mu^*(v)\ge 3 +1/2-2(1/2)=5/2$.  So $v$ has a medium neighbor.  If $v$ has
two or more medium neighbors, then $\mu^*(v)\ge 3 - 1/2=5/2$.  So suppose $v$
has one medium neighbor and two 2-neighbors.  If $v$ has two incident 1-threads,
then $\mu^*(v)\ge 3-2(1/4)=5/2$.  So now assume that $v$ has an incident
2-thread and the other incident thread is either a 2-thread or a 1-thread
leading to a medium weak neighbor.
Let $u_1$ be a neighbor of $v$ on a 2-thread and let $u_2$ be the other 2-neighbor
of $v$. See Figure~\ref{f.5}$(iv)$. By minimality, we can color
$G\setminus\{u_1,u_2,v\}$.  Now we color $v$, $u_2$, $u_1$ in this order.

Finally, suppose $d(v)=2$.  We show that each $\l$-thread $P$ receives charge at
least $\l/2$, so that each 2-vertex finishes with charge $5/2$.  If $\l=3$, then
$P$ receives $1/2$ from each endpoint and an additional $1/2$ from its sponsor.
If $\l=2$, then $P$ receives charge $1/2$ from each endpoint.  If $\l=1$, then
either $P$ receives charge $1/2$ from a high endpoint or $P$ receives charge
$1/4$ from both medium endpoints.

Thus, each vertex finishes with charge at least $5/2$.  This contradicts the
fact that $\mad(G)<5/2$, and thus completes the proof.
\end{proof}

%
%
%

\section{Maximum degree 7}

\begin{thm}
\label{T6}
If $\Delta\le 7 $ and $\mad(G)<2+20/37$, then $\chil(G^2)\le 8$.
\end{thm}
\begin{proof}
Assume to the contrary that the theorem is false and let $G$ be a minimal
counterexample.  Again, we use discharging.  A vertex is \emph{high} if
its degree is 6 or 7, it is \emph{medium} if its degree is 4 or 5, and
it is \emph{low} otherwise.

By reducible configurations (C1) and (C2), $G$ contains no 4-thread and each
3-thread is incident to $7$-vertices at both ends.  By (C4), the subgraph
induced by 3-threads is acyclic.  Hence, we can assign each $7$-vertex
to sponsor at most one incident 3-thread so that every 3-thread is sponsored.

We use discharging with each vertex $v$ getting initial charge $d(v)$ and
with the following discharging rules.
\begin{enumerate}
\item[(R1)] Every high vertex sends charge $21/37$
in each direction.
\item[(R2)] Every 7-vertex sends extra charge $18/37$ to its sponsored 3-thread (if
it exists).
\item[(R3)] Every 5-vertex sends charge $20/37$ to each incident 2-thread, and
it sends charge $10/37$ in each direction that does not lead to a 2-thread.
\item[(R4)] Every 4-vertex sends charge $20/37$ to each incident 2-thread,
$10/37$ to each incident 1-thread, and $4/37$ to each adjacent 3-vertex.
\item[(R5)] Every 3-vertex sends charge $19/37$ to each incident 2-thread, and
$10/37$ to each incident 1-thread leading to a weak $5^-$-neighbor.
\end{enumerate}

We now show that every vertex finishes with charge $\mu^*(v)$ at least $2+20/37$.

\begin{figure}[tb]
\label{f.7}
\begin{center}
   \begin{tabular}{ccc}
     \includegraphics[scale=1]{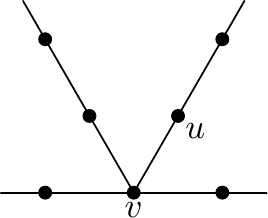}\hspace{.5cm}&
     \includegraphics[scale=1]{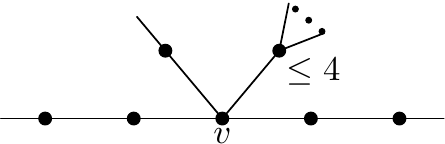}\hspace{.5cm}&
     \includegraphics[scale=1]{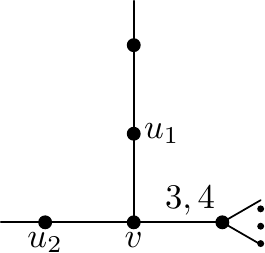}\\
     $(i)$ & $(ii)$ & $(iii)$ \\
   \end{tabular}

   \caption{Reducible configurations from Theorem~\ref{T6}}
\end{center}
\end{figure}

If $d(v)=7$, then $\mu^*(v)\ge 7-7(21/37)-(18/37)=2+20/37$.  If $d(v)=6$, then
$\mu^*(v)\ge 6-6(21/37)=2+22/37$.
If $d(v)=5$, then $\mu^*(v)\ge 5-4(20/37)-1(10/37)=2+21/37$ unless $v$ has five
incident 2-threads.  In this case $N[v]$ is
reducible.  By minimality, we color $G\setminus N[v]$. Then we
easily color $v$, followed by its neighbors.

Suppose $d(v)=4$. If $v$ has at most one incident 2-thread, then $\mu^*(v)\ge
4-20/37-3(10/37)=2+24/37$.  Similarly, if $v$ has two incident 2-threads and
only one incident 1-thread, then $\mu^*(v)\ge4-2(20/37)-10/37-4/37=2+20/37$. 
So either $v$ has two incident 2-threads and two incident 1-threads or $v$ has
at least three incident 2-threads.
Consider the first case, shown in Figure~\ref{f.7}$(i)$. Let $u$ be a neighbor
of $v$ on a 2-thread. By minimality, we color $G-u$. Now we uncolor $v$, then
color $v$ and $u$, in that order. Consider now the second case, where $v$ has
at least three incident 2-threads.  If $v$
has a $5^+$-neighbor, then $\mu^*(v)\ge 4+10/37-3(20/37)=2+24/37$.  If
$v$'s fourth neighbor is instead a $4^-$-vertex, then color $G-v$ by
minimality. See Figure~\ref{f.7}$(ii)$. Uncolor three neighbors of $v$ on
2-threads; now color $v$, followed by its uncolored neighbors.

Suppose $d(v)=3$.
If $v$ has three 2-neighbors and at least one of them, $u$, is on a 2-thread,
then color $G-u$ by minimality; recolor $v$ to avoid the color on $u$'s other
neighbor, then color $u$.  If $v$ has three incident 1-threads and each
leads to a weak $6^+$-neighbor, then $\mu^*(v)=3-0>2+20/37$.  If $v$ has three
incident 1-threads and one of them, with internal vertex $u$, leads to a weak
$5^-$-vertex, then color $G-u$ by minimality.  Now color $u$, then recolor $v$.
 So $v$ has at most two 2-neighbors.

If $v$ has exactly one 2-neighbor and it lies on a 1-thread, then $\mu^*(v)\ge
3-10/37 = 2 + 27/37$.
Suppose $v$ has exactly one 2-neighbor and it lies on a 2-thread.
If either of its other neighbors is a $4^+$-vertex, then $\mu^*(v)\ge
3-19/37+4/37=2+22/37$.  Otherwise, let $u$ denote the 2-neighbor.  By
minimality, color $G-u$.  Recolor $v$ to avoid the color on $u$'s other
neighbor, then color $u$.  So $v$ must have exactly two 2-neighbors.

Suppose that $v$ has exactly two 2-neighbors, $u_1$ and $u_2$.
If $v$'s third neighbor $w$ is a high vertex, then $\mu^*(v)\ge
3+21/37-2(19/37)=2+20/37$, so $w$ must be a $5^-$-vertex.
If both $u_i$ lie on 2-threads, then color $G\setminus\{v,u_1,u_2\}$.  Now
color $v$, $u_1$, then $u_2$.
Suppose instead that $v$ has exactly one incident 2-thread and that $u_1$ is
its neighbor on the 2-thread.  If $v$'s $3^+$-neighbor $w$ is a
$4^-$-vertex, then color $G-u_1$ by minimality; see Figure~\ref{f.7}$(iii)$. Now,
recolor $v$ to avoid the color on the other neighbor of $u_1$, then color $u_1$;
so $w$ must be a 5-vertex.
If $v$'s weak neighbor along the 1-thread is a $6^+$-vertex, then $v$ sends no
charge along the 1-thread so $\mu^*(v)\ge 3+10/37-19/37=2+28/37$.  Otherwise
this weak neighbor via the 1-thread containing $u_2$ is a $5^-$-vertex.  Color
$G\setminus\{v,u_1,u_2\}$ by minimality; now color $v$, $u_2$, and $u_1$, in
that order.

Suppose $v$ has exactly two incident 1-threads.  If $v$'s third neighbor is a
$4^+$-vertex, then $\mu^*(v)\ge 3+4/37-2(10/37)=2+21/37$.  Similarly, if either
of $v$'s weak neighbors is high, then $\mu^*(v)\ge 3-10/37 =2+27/37$.  Thus,
assume that $v$'s 2-neighbors $u_1$ and $u_2$ each lead to weak $5^-$-neighbors
and that it's third neighbor is a $3$-vertex.  Color
$G\setminus\{v,u_1,u_2\}$ by minimality.  We can now color $u_1$, $u_2$, and
$v$, in that order.

Finally, suppose $d(v)=2$.  We show that each $\l$-thread $P$ receives charge at
least $20\l/37$, so that each 2-vertex finishes with charge $2+20/37$.  If $\l=3$, then
$\mu^*(P)=6+2(21/37)+18/37=3(2+20/37)$.
If $\l=2$ and one endpoint of $P$ is a 3-vertex, then by reducible configuration
(C3) the other endpoint must be a 7-vertex, so $\mu^*(P)=4+21/37+19/37=2(2+20/37)$.
If $\l=2$ and both endpoints of $P$ are $4^+$-vertices, then $\mu^*(P)\ge4+2(20/37)=2(2+20/37)$.
If $\l=1$, then $\mu^*(P)\ge 2+2(10/37)=2+20/37$.

Thus, each vertex finishes with charge at least $2+20/37$.  This contradicts the
fact that $\mad(G)<2+20/37$, and so completes the proof.
\end{proof}

%
%
%

\section{Maximum degree at least 8}

\begin{thm}
\label{thm8}
For $k\ge 8$, if $\Delta\le k$ and $\mad(G)<2+\frac{4k-8}{5k+2}$, then
$\chil(G^2)\le k+1$.
\end{thm}
\begin{proof}
Suppose to the contrary that the theorem is false and let $G$ be a minimal
counterexample.  A $3^+$-vertex is \emph{high} if its degree is $k$ or $k-1$, it is
\emph{medium} if its degree is between $k-2$ and $7-\floor{\frac{16}{k+2}}$
(inclusive), and otherwise it is \emph{low}.
By reducible configurations (C1) and (C2), $G$ contains no 4-thread and each
3-thread is incident to $k$-vertices at both ends.  By configuration (C4), the
subgraph induced by 3-threads is acyclic.  Hence, we can assign each $k$-vertex
to sponsor at most one incident 3-thread so that every 3-thread is sponsored.

Let $\alpha=\frac{4k-8}{5k+2}$.  Let $\beta=\frac{(k-2)-4\alpha}{k-2}=1-\frac{16}{5k+2}$.
For $k\ge 8$, note that $\alpha/2>8/(5k+2)$ and $1 > \beta > \alpha > 2\alpha - \beta >
2/5 > \alpha/2 > 1/5 > \beta-\alpha$. (Verifying these inequalities is tedious,
but straightforward.)  The following equality also holds:
\begin{equation}  \label{i.3}
                2-16/(5k+2)= 5\alpha - 2\beta.
\end{equation}

We use discharging with each vertex $v$ getting initial charge $d(v)$ and with
the following discharging rules.

\begin{enumerate}
\item[(R1)] Every high vertex sends charge $\beta$ in each direction.
\item[(R2)] Every $k$-vertex sends charge $3\alpha-2\beta$ to its sponsored 3-thread (if it
exists).
\item[(R3)] Every medium vertex sends charge $2\alpha-\beta$ in each direction.
\item[(R4)] Every low vertex $v$ sends charge $2\alpha-\beta$ to each incident
2-thread (leading to a weak $k$-neighbor, by (C3), since $v$ is by definition a
$(k-2)^-$-vertex), $\alpha/2$ to each incident 1-thread leading to a low weak
neighbor, and $\beta-\alpha$ to each incident 1-thread leading to a medium weak
neighbor.
\item [(R5)] Every 3-vertex receives charge $8/(5k+2)$ from each adjacent 5-vertex and
charge $4/(5k+2)$ from each adjacent 4-vertex.\footnote{When $k\le 10$, in
some cases 3-vertices need more charge; that is the point of this rule.  When
$k\ge 11$ this rule is not needed, but since $8/(5k+2)$ and $4/(5k+2)$ rapidly
diminish to 0, this rule causes no problems.}
\end{enumerate}

We now show that every vertex finishes with charge at least $2+\alpha$.
This will contradict the fact that $\mad(G)<2+\alpha$, and thus prove the
theorem.

\medskip

\noindent
\textbf{Case:} $d(v)=k$. Now $\mu^*(v)\ge
k-k\beta-(3\alpha-2\beta)=k-(k-2)\beta-3\alpha=k-((k-2)-4\alpha)-3\alpha=2+\alpha$.

\medskip

\noindent
\textbf{Case:} $d(v)=k-1$. Now
$\mu^*(v)\ge (k-1)-(k-1)\beta=(k-1)-\beta-((k-2)-4\alpha)=1+4\alpha-\beta\ge 2+\alpha$.  

\medskip

\noindent
\textbf{Case:} {\em $v$ is medium}. Let $d=d(v)$. Now $\mu^*(v)\ge d-d(2\alpha-\beta)$.  This quantity is at least $2+\alpha$ when $d\ge (2+\alpha)/(1+\beta-2\alpha) = (7k-2)/(k+2)=7-\frac{16}{k+2}$.

\medskip

\noindent
\textbf{Case:} $d(v)=6$ {\em and $v$ is low}. If $v$ has at most five
2-neighbors, then $\mu^*(v)\ge 6-5(2\alpha-\beta) \ge 2+\alpha$; note that the
last innequlity is is equivalent to $4\ge 11\alpha - 5\beta$, i.e.\ it is
equivalent to $(26+k)/(2+5k)\ge0$.  If $v$ has at most four incident 2-threads,
then $\mu^*(v)\ge 6-4(2\alpha-\beta)-2(\alpha/2) \ge 2+\alpha$ by~(\ref{i.3}).
So now $v$ must have six 2-neighbors and at least five incident 2-threads.  Form $H$
from $G$ by deleting $v$ and each of its neighbors on a 2-thread.  Color $H$ by
minimality.  Now we can color $v$ (it has at most 7 restrictions on its color).
 Finally we can color each uncolored neighbor of $v$
(the last has at most 8 restrictions on its color).

\medskip

\noindent
\textbf{Case:} $d(v)=5$.
If $v$ has at most four 2-neighbors and at most three of them lie on 2-threads, then
$\mu^*(v)\ge 5 - 3(2\alpha-\beta) -\alpha/2 -8/(5k+2) \ge 2+\alpha$. This last inequality
simplifies to $16/(2+5k)\ge 0$, which obviously holds.
Similarly, if $v$ has five  2-neighbors at most one of which lies on a 2-thread, then
$\mu^*(v)\ge 5 -(2\alpha-\beta)-4(\alpha/2) \ge 2+\alpha$ and the last
inequality simplifies to $32/(2+5k)\ge 0$.  If $v$ has five
2-neighbors and at least two of them lie on 2-threads, then let $u_1$ and $u_2$ be
neighbors on 2-threads.  By minimality, color $G\setminus\{u_1,u_2\}$.  Now
recolor $v$ to avoid the colors on the neighbors of $u_1$ and $u_2$ ($v$
has at most 8 constraints on its color), then color $u_1$ and $u_2$.  Hence
$v$ must have exactly four 2-neighbors, all of which lie on 2-threads. If the final
neighbor of $v$ is a medium vertex, then it gives $v$ charge $2\alpha-\beta$.  So
$\mu^*(v)\ge 5-4(2\alpha-\beta)+(2\alpha-\beta)\ge 2+\alpha$; note that the
last inequality simplifies to $3\ge 7\alpha - 3\beta$ and is equivalent to
$2(10+k)/(2+5k)\ge 0$, which obviously holds for all $k$'s we consider.
So we can assume the final neighbor $u$ of $v$ is a $6^-$-vertex and gives no
charge to $v$.  For $k\le 14$, we get $\mu^*(v)\ge 5-4(2\alpha-\beta)-8/(5k+2)$,
which is bigger than $2+\alpha$  since $3-8/(5k+2)\ge 9\alpha - 4\beta$.
And for $k\ge 15$, we proceed as follows.
Color $G\setminus (N[v]-u)$ by minimality.  Now color $v$ (it has at most
10 constraints on its color), then color each uncolored neighbor of $v$ (the
last has at most 7 constraints on its color).

\begin{figure}
\label{f.8}
\begin{center}
  \begin{tabular}{cc}
    \includegraphics[scale=1]{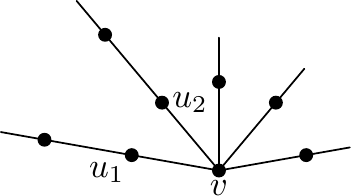}\hspace{.4cm}&\hspace{.4cm}
    \includegraphics[scale=1]{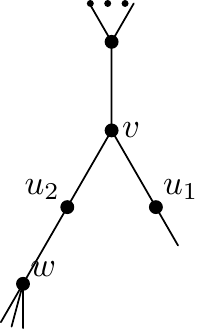}\\
     $(i)$ & $(ii)$
  \end{tabular}
 \caption{Reducible configurations from Theorem~\ref{thm8}}
\end{center}
\end{figure}

\medskip

\noindent
\textbf{Case:} $d(v)=4$.
Suppose $v$ has four 2-neighbors.  If $v$ has an incident 2-thread, then let $u$
denote $v$'s neighbor on the 2-thread.  By minimality, color $G-u$.
Now recolor $v$ to avoid the color on $u$'s other neighbor ($v$ has
at most 7 restrictions on its color), then color $u$ (which has at most 6
restrictions on its color).  So assume instead that $v$ has four incident 1-threads.
Now $\mu^*(v)\ge 4 - 4(\alpha/2)$.  This quantity is at least $2+\alpha$ when $k\le 14$.
So suppose $k\ge 15$.  If $v$ has a 1-thread leading to a low weak neighbor,
then let $u$ denote the neighbor on this 1-thread.  By minimality, color $G-u$.
Now recolor $v$ to avoid the color on $u$'s other neighbor ($v$ has
at most 7 restrictions on its color), then color $u$ (which has at most 10
restrictions).  So suppose instead that each of the four 1-threads
incident to $v$ leads to a medium or high weak neighbor.  Now $\mu^*(v)\ge 4 - 4
(\beta-\alpha) \ge 2+\alpha$, where the last inequality is equivalent to
$2(18+k)/(2+5k)\ge0$.
Thus, $v$ has at most three 2-neighbors.

If $v$ has a medium or high neighbor, then $\mu^*(v)\ge 4 - 3(2\alpha-\beta)+(2\alpha-\beta) \ge 2+\alpha$
by~(\ref{i.3}).  Similarly, if $v$ has at most two 2-neighbors, then $\mu^*(v)\ge 4 -2(2\alpha-\beta)-2(4/(5k+2))\ge 2+\alpha$ by~(\ref{i.3}).  Thus,
$v$ has exactly three 2-neighbors and one low neighbor.
If $v$ has three incident 2-threads and its low neighbor $u$ is a $5^-$-neighbor, then
color $G\setminus (N[v]-u)$ by minimality.  Now color $v$ (it has at most 8
restrictions on its color), then color each of its uncolored neighbors (the last
has at most 6 restrictions on its color).  Suppose instead that $v$ has three
incident 2-threads and its low neighbor is a 6-vertex; so $k\ge 15$.  Now $v$
has at most 9 restrictions on its color, but since $k\ge 15$, we can
complete the coloring.

So we may assume that $v$ has at most two incident 2-threads.
If $v$ has no incident 2-threads, then $\mu^*(v)\ge 4-3(\alpha/2)-4/(5k+2)\ge
2+\alpha$, which is easy to verify.  If $v$ has one incident 2-thread, then $\mu^*(v)\ge
4-(2\alpha-\beta)-2(\alpha/2)-4/(5k+2)\ge \alpha +2$ simplifies to
$1 \ge 7\alpha/2 - 2\beta$, which holds when $k\le 18$.
When $k\ge 19$, let $u$ be $v$'s neighbor on its 2-thread.
Color $G-u$ by minimality.  Recolor $v$ to avoid the color on $u$'s other
neighbor ($v$ has at most $6+2+2+1$ constraints on its color), then color
$u$.  So, suppose $v$ has two incident 2-threads.  If $v$'s $3^+$-neighbor is a
3-neighbor, then let $u$ be a neighbor on a 2-thread.  By minimality, color
$G-u$.  Recolor $v$ to avoid the color on $u$'s other neighbor ($v$ has
at most 8 restrictions on its color), then color $u$.  If instead, $v$'s final
neighbor is a $4^+$-vertex, then $\mu^*(v)\ge 4-2(2\alpha-\beta)-\alpha/2$.
This quantity is at least $2+\alpha$ when $k\le 10$.  So assume $k\ge 11$ and let
$u$ be a neighbor on a 2-thread.  Color $G-u$ by minimality.  Now recolor
$v$ to avoid the color on $u$'s other neighbor ($v$ has at most 11
restrictions on its color), then color $u$.

\medskip

\noindent
\textbf{Case:} $d(v)=3$.
If $v$ has no 2-neighbors, then $v$ begins with charge 3 and gives away no
charge, so assume that $v$ has at least one 2-neighbor.  First suppose that
$v$ has three 2-neighbors and at least one such neighbor $u$ is on a
2-thread or on a 1-thread leading to a medium or low weak neighbor.
Color $G-u$ by minimality.  Uncolor $v$ and color $u$ (it has at most
$(k-2)+2$ constraints on its color); now color $v$ (which has at most 6
constraints on its color).  If instead $v$ has three incident 1-threads and each
leads to a high weak neighbor, then $\mu^*(v)=3-0$.  So $v$ must have at
most two 2-neighbors.

Suppose that $v$ has exactly two 2-neighbors; call them $u_1$ and $u_2$.  If $u_1$
and $u_2$ both lie on 2-threads and $v$'s third neighbor is a medium or low
neighbor, then color $G\setminus\{v,u_1,u_2\}$ by minimality.  Now color $v$
(which has at most $(k-2)+2$ constraints on its color), then $u_1$ and $u_2$.
If instead $v$ has a high neighbor, then
$\mu^*(v)=3-2(2\alpha-\beta)+\beta=3-4\alpha+3\beta\ge 2+\alpha$ which holds
by~(\ref{i.3}) since $\beta = 1 - 16/(5k+2)$.
So $v$ must not have two incident 2-threads.

Suppose that $v$ has an incident 1-thread and an incident 2-thread, with $u_1$
on the 2-thread.  If $v$'s $3^+$-neighbor is low, then color $G-u_1$ by
minimality.  Recolor $v$ to avoid the color on $u_1$'s other neighbor.
If $v$'s $3^+$-neighbor has degree at most 5, then $v$ has at most $5+2+1$
constraints on its color; if it has degree 6, then $v$ has at most $6+2+1$
constraints on its color, but now $k\ge 15$ since a 6-vertex is low, so the
recoloring succeeds.  Finally, color $u_1$.  So $v$'s $3^+$-neighbor is medium
or high; if it is high, then $\mu^*(v)\ge 3+\beta-(2\alpha-\beta)-\alpha/2\ge 2
+ \alpha$ by $(\ref{i.3})$.  So assume $v$'s $3^+$-neighbor is medium.  Let
$u_2$ be the neighbor of $v$ on a 1-thread and $w$ the other neighbor of $u_2$;
see Figure~\ref{f.8}$(ii)$.  If $w$ is medium or high, then $\mu^*(v)\ge
3+(2\alpha-\beta)-(2\alpha-\beta)-(\beta-\alpha)\ge 2+\alpha$, since the second
inequality is equivalent to $1\ge \beta$.  So assume $w$ is low.  Now color
$G\setminus\{v,u_1,u_2\}$ by minimality.  Color $v$ (it has at most
$(k-2)+1+1$ constraints on its color), then $u_2$, and finally $u_1$.

Suppose that $v$ has two incident 1-threads.
If $v$'s $3^+$-neighbor is medium or high, then $\mu^*(v)\ge
3+(2\alpha-\beta)-2(\alpha/2)\ge 2+\alpha$, which holds
since $1 > \beta$.  So assume this third neighbor is
low. If at least one 1-thread leads to a weak neighbor that is low, then let $u_1$ be
the 2-vertex on that 1-thread.  Color $G-u_1$ by minimality.
Recolor $v$ to avoid the color on the neighbor of
$u_1$; if $v$'s $3^+$-neighbor has degree at most 5, then $v$ has at most $5+2+1$
constraints on its color, and if it has degree 6, then $v$ has at most $6+2+1$
constraints on its color, but now $k\ge 15$ since a 6-vertex is low, so the
recoloring succeeds.  Now color $u_1$ (the analysis of constraints on the color
of $u_1$ is analogous to that of $v$).  So neither 1-thread leads to a low weak
neighbor.

If at least one 1-thread leads to a weak neighbor that is high, then $\mu^*(v)\ge
3-(\beta-\alpha) \ge 2 + \alpha$, which holds since $1> \beta$.
So assume that both 1-threads
lead to weak neighbors that are medium.  Now $\mu^*(v)\ge 3 - 2(\beta-\alpha)$.
This quantity is at least $2+\alpha$ when $k\le 22$; so assume $k\ge 23$.  Color
$G\setminus\{v,u_1,u_2\}$ by minimality.  Now color $u_1$ and $u_2$ (each has
at most $(k-2)+1+1$ constraints on its color), then color $v$ (which
has at most $6+2+2$ constraints on its color).  Thus $v$ must have exactly
one 2-neighbor.

Suppose that $v$ has one 2-neighbor and two $3^+$-neighbors.  If either
$3^+$-neighbor is medium or high, then $\mu^*(v)\ge 3 + (2\alpha-\beta)-
(2\alpha-\beta)=3> 2 + \alpha$; so assume that both $3^+$-neighbors are low
vertices.  Suppose that the 2-neighbor $u$ lies on a 1-thread.
(We assume that $v$'s weak neighbor is medium or low, since otherwise
$\mu^*(v)=3-0$.) Now $\mu^*(v)\ge 3 - \alpha/2$.  This quantity is at least
$2+\alpha$ when $k\le 14$.  So suppose $k\ge 15$.  Now color $G-u$ by
minimality, and uncolor $v$.  First color $u$ (which has at most $(k-2)+1+1$
constraints on its color), then color $v$ (which has at most $6+6+2$
constraints). 

So now suppose that $v$'s single 2-neighbor $u$ lies on a 2-thread.  As above,
$v$'s two other neighbors must be low.  Suppose $k\ge 13$.  By minimality, color
$G-u$.  Recolor $v$ to avoid the color on $u$'s other neighbor (it has at
most $6+6+1$ constraints on its color), then color $u$ (which has at most
5 constraints).  Similarly, if $11\le k\le 12$, then 6-vertices
are medium.  In this case, the same recoloring process works; now however $v$'s
low neighbors must be $5^-$-vertices, so $v$
has at most $5+5+1$ constraints on its color.  Now suppose $k\le 10$.
If $v$ has a 5-neighbor, then $\mu^*(v)\ge 3-(2\alpha-\beta)+8/(5k+2)$.  This
quantity is at least $2+\alpha$ when $k\le 10$.
If $v$ has a 4-neighbor, then $\mu^*(v)\ge 3-(2\alpha-\beta)+4/(5k+2)$.  This
quantity is at least $2+\alpha$ when $k=8$.  
If $v$ has two 3-neighbors, then we again use the recoloring process above.
Finally, if $9\le k\le 10$ and $v$
has no $5^+$-neighbor, then the recoloring process above works again; this
time $v$ has at most $4+4+1$ constraints on its color (which is fine, since
$k\ge 9$).

\medskip

\noindent
\textbf{Case:} $d(v)=2$.  We show that each $\l$-thread $P$ receives charge at
least $\alpha \l$, so that each 2-vertex finishes with charge at least
$2+\alpha$.  If $\l=3$, then $\mu^*(P)=6+2\beta+(3\alpha-2\beta)=3(2+\alpha)$.
If $\l=2$, then by reducible configuration (C3) at least one endpoint of $P$
must be a high vertex.  So $\mu^*(P)\ge4+\beta+(2\alpha-\beta)=2(2+\alpha)$.
Finally, suppose $\l=1$.  If at least one endpoint is a medium or high vertex,
then $\mu^*(P)\ge 2+(2\alpha-\beta)+(\beta-\alpha)=2+\alpha$.  If instead both
endpoints are low vertices, then $\mu^*(P)\ge 2 + 2(\alpha/2)=2+\alpha$.
\bigskip

We have shown that every vertex finishes with charge at least $2+\alpha$.  This
contradicts the fact that $\mad(G)<2+\alpha$, and so finishes the proof.
\end{proof}

Thanks very much to two anonymous referees, whose careful reading of the
paper was unusually helpful.

\end{document}